\documentclass[11pt]{amsart}
\usepackage[hmargin=3cm,vmargin=3cm]{geometry}
\usepackage{combelow} 

\usepackage[latin1,utf8]{inputenc} 
\usepackage[english]{babel}
\usepackage{etex}

\usepackage{amssymb,amsfonts,amsthm,amsmath,amscd,latexsym}
\usepackage{palatino, euscript}
\usepackage{graphicx}
\usepackage[all]{xy}
\SelectTips{cm}{}

\usepackage{tikz, tikz-cd}
\usetikzlibrary{matrix, arrows, calc}
\tikzset{frontline/.style={preaction={draw=white,-,line width=6pt}},}  

\usepackage[dvips]{epsfig}
\usepackage{pinlabel}

\newcommand{\ig}[2]{\vcenter{\xy (0,0)*{\includegraphics[scale=#1]{fig/#2}} \endxy}}

\usepackage{enumitem}

\newtheorem{thm}{Theorem}[section]

\newtheorem{cor}[thm]{Corollary}

\newtheorem*{prop*}{Proposition}

\theoremstyle{definition}
\newtheorem{defn}[thm]{Definition}

\theoremstyle{remark}

\newtheorem{rem}[thm]{Remark}

\numberwithin{equation}{section}

\def\CC{{\mathcal{C}}}

    \def\HC{{\mathcal{H}}}

\let\phi=\varphi

\usepackage{bbm}

\def\Z{{\mathbbm Z}}

\def\one{\mathbbm{1}}

\newcommand{\ot}{\otimes}
\newcommand{\pa}{\partial}
\newcommand{\co}{\colon}

\renewcommand{\to}{\rightarrow}

\newcommand\simto{\xrightarrow{
   \,\smash{\raisebox{-0.65ex}{\ensuremath{\scriptstyle\sim}}}\,}}

\newcommand{\refequal}[1]{\xy {\ar@{=}^{#1}
(-1,0)*{};(1,0)*{}};
\endxy}

\DeclareMathOperator{\Hom}{Hom}

\DeclareMathOperator{\End}{End}

\DeclareMathOperator{\Sym}{Sym}

\renewcommand{\1}{\mathbbm{1}}

\newcommand{\startdotred}{\ig{.5}{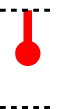}}

\newcommand{\finaldotred}{\ig{.5}{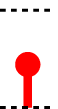}}

\newcommand{\splitred}{\ig{.5}{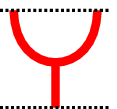}}
\newcommand{\mergered}{\ig{.5}{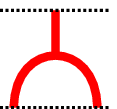}}
\newcommand{\pitchred}{\ig{.5}{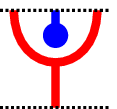}}
\newcommand{\pitchdownred}{\ig{.5}{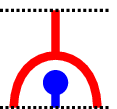}}
\newcommand{\pitchblue}{\ig{.5}{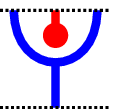}}
\newcommand{\pitchdownblue}{\ig{.5}{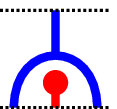}}

\newcommand{\sbotttop}{{
\labellist
\tiny\hair 2pt
\pinlabel {$\dots$} [ ] at 20 27
\pinlabel {$\dots$} [ ] at 20 4
\endlabellist
\centering
\ig{1}{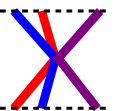}
}}

\begin{document}

\begin{abstract} The Hecke category is bigraded. For completeness, we classify gradings on the Hecke category. We also classify object-preserving autoequivalences. \end{abstract}

\title{The Hecke category is bigraded}

\author{Ben Elias} \address{University of Oregon, Eugene}

\maketitle

\setcounter{tocdepth}{1}
\tableofcontents

This is a short note written for the experts, so we cut to the chase. We assume the reader is familiar with the diagrammatic Hecke category $\HC$ associated to a realization $V$ of
a Coxeter system $(W,S)$, over a commutative base ring $\Bbbk$. An explicit presentation can be found in \cite[\S 10.2.4]{EMTW}. Typically, this is defined as a $\Z$-graded
category, where e.g. univalent vertices have degree $1$, and elements of $V$ have degree $2$ when viewed within the polynomial ring $R = \Sym(V)$. Note that $R \cong
\End(\1)$, where $\1$ is the monoidal identity.

{\bf Acknowledgments} We wish to thank Matt Hogancamp. The author was supported by NSF grants DMS-2201387 and DMS-2039316.

\section{The bigrading}

The new observation is that $\HC$ can actually be bigraded. We believe this minor observation will prove to be quite significant, but for now provide only some brief remarks.

\begin{thm} \label{thm:bigraded} The diagrammatic Hecke category $\HC$ (if it is well-defined, see below) can be defined as a $\Z^2$-graded monoidal category. The bidegrees of the generating morphisms are as follows.
\begin{subequations} \label{eq:bigrading}
\begin{equation} \deg \left( \startdotred \right) = (1,0), \qquad \deg \left( \finaldotred \right) = (0,1), \qquad \deg \left( \splitred \right) = (0,-1), \qquad 
	\deg \left( \mergered \right) = (-1,0), \end{equation}
\begin{equation} \label{eq:degofmst} \deg \left( \sbotttop \right) = (0,0). \end{equation}
\end{subequations}
The picture in \eqref{eq:degofmst} represents an arbitrary $2m_{st}$-valent vertex. The degree of any element of $V$, viewed as a linear polynomial in $R$, is $(1,1)$. The original degree function is obtained from the bidegree via the map
\begin{equation} \Z^2 \to \Z, \qquad (m,n) \mapsto m+n. \end{equation}
\end{thm}

All proofs of theorems can be found in \S\ref{s:proofs}.

\begin{rem} Note that cups and caps have bidegree $(1,-1)$ and $(-1,1)$ respectively. The bidegree of a morphism is an isotopy invariant relative to a fixed boundary, but the bidegree is not preserved under adjunction isomorphisms.  \end{rem}

\begin{rem} We expect that the second degree is in some sense the Soergel-ification of the weight filtration. More precisely, the Frobenius automorphism in geometry should
correspond to rescaling based on the second degree. We give more evidence in a paper currently in preparation. Note that Soergel bimodules obfuscate the weight filtration, as they
arise from semisimple (and thus pure) perverse sheaves. However, the weight filtration should be visible on non-semisimple objects, of which Rouquier complexes are an example.
Indeed, the bigrading above does break the symmetry inherent in Rouquier complexes for positive and negative braids. \end{rem}

\begin{rem} This bigrading may equip Khovanov-Rozansky HOMFLYPT homology with a fourth grading. \end{rem}

\begin{rem} This bigrading extends to the 2-category known as the singular (diagrammatic) Hecke category, which describes singular Soergel bimodules, see \cite[\S 24]{EMTW}. The
idea that the singular Hecke category could be bigraded originated out of \cite[\S 3.4]{EKo}, where two separate length functions $\ell^-$ and $\ell^+$ are defined on double
cosets. \end{rem}

\begin{rem} The diagrammatic Hecke category $\HC$ is not currently defined in types containing $H_3$, because the appropriate Zamolodchikov relation is unknown. Presumably, when
found, it will be compatible with the bigrading. Additional requirements for the well-definedness of $\HC$ are found in \cite[\S 5.1]{EWLocalized}, though see also the simplification
found in \cite{HaziRotatable}. \end{rem}

\section{Classifying gradings}

For sake of completeness, let us discuss all possible multigradings on $\HC$ which refine the original grading. It is assumed that the ground ring $\Bbbk$ has degree zero. For sake of simplicity, we first assume that $V$ is the root realization.

\begin{thm} \label{thm:multigraded} Suppose $V$ is the root realization. Let $\Lambda$ be the free abelian group generated by symbols $\{f_s, g_s\}_{s \in S}$, modulo the relation that $f_s + g_s = f_t + g_t$ whenever $m_{st} \ne 2$. Then $\HC$ admits a $\Lambda$-grading, as below. Here $s$ is red and $t$ is blue, and purple strands are either red or blue.
\begin{subequations}  \label{eq:multigrading}
\begin{equation} \label{eq:onecolor} \deg \left( \startdotred \right) = f_s, \qquad \deg \left( \finaldotred \right) = g_s, \qquad \deg \left( \splitred \right) = -g_s, \qquad 
	\deg \left( \mergered \right) = -f_s, \end{equation}
\begin{equation} \label{eq:2mvalent} \deg \left( \sbotttop \right) = \begin{cases} 0 & \text{ if } m_{st} \text{ is even}, \\ g_s - g_t & \text{ if } m_{st} \text{ is odd}. \end{cases} \end{equation}
\end{subequations}
Moreover, this is the universal monoidal grading which refines the original grading.
\end{thm}

Note that $f_s + g_s = f_t + g_t$ whenever $s$ and $t$ are in the same connected component of the Coxeter graph. Abstractly, $\Lambda$ is a free abelian group of rank $n + k$, where
$n$ is the number of simple reflections, and $k$ is the number of irreducible components of the Coxeter system.

\begin{rem} Moreover, $\Lambda$ can be viewed as the product of the universal grading group for each component. There is a general process whereby, given monoidal categories
$\CC_i$, one can construct formally a monoidal category $\CC = \bigotimes \CC_i$ which contains $\CC_i$ for each $i$, and for which $\CC_i$ tensor-commutes (functorially) with
$\CC_j$ when $j \ne i$. If $\CC_i$ is graded by abelian groups $\Lambda_i$, then $\CC_i$ is graded by $\Lambda = \prod \Lambda_i$. In our case $\HC$ is isomorphic to $\bigotimes
\HC_i$, where $\HC_i$ is the Hecke category of each irreducible component, at least when working with the root realization. \end{rem}

Here is a generalization which deals with other realizations.

\begin{thm} \label{thm:mostgeneral} Let $V$ be a realization. Equip $V$ with a $W$-invariant grading, valued in an abelian group $\Gamma$, for which roots are homogeneous, and
$\deg(\alpha_s) = \deg(\alpha_t)$ whenever $m_{st} \ne 2$. Then this grading on $V \subset R$ extends to a grading on $\HC$ by the
group $(\Lambda \times \Gamma) / I$, where $I$ is the submodule generated by $(f_s + g_s - \deg(\alpha_s))$. \end{thm}

The bigrading of Theorem \ref{thm:bigraded} is a specialization of the one in Theorem \ref{thm:mostgeneral}, when all of $V$ has degree $1 \in \Gamma = \Z$.

\section{Classifying autoequivalences}

\begin{defn}
Let $\CC$ be an additive $\Bbbk$-linear category which is graded by an abelian group $\Lambda$. Let $\chi : \Lambda \to \Bbbk^\times$ be a group homomorphism. Then $\CC$ admits an autoequivalence $\Theta_{\chi}$ which fixes all objects, and acts on homogeneous morphisms by the formula
\begin{equation} \Theta_{\chi}(f) = \chi(\deg(f)) \cdot f. \end{equation}
We call such an autoequivalence a \emph{degree-based rescaling}.
\end{defn}

It is straightforward to prove that degree-based rescalings are well-defined equivalences. Moreover, if $\CC$ is graded monoidal, then $\Theta_{\chi}$ is a monoidal equivalence,
as the monoidal composition is also compatible with the grading.

Consequently, $\HC$ admits a multi-parameter family of monoidal degree-based rescalings. Even though these autoequivalences were produced by considering more general
gradings on $\HC$, they can be viewed as autoequivalences of the original, singly-graded category $\HC$. The two-parameter family of autoequivalences coming from the bigrading of Theorem \ref{thm:bigraded} has been known to the experts for some time. For example, a special case was used in \cite[Lemma 4.18]{MackaayTubbenhauer}.

It should not be terribly surprising that autoequivalences which fix objects might be in bijection with gradings by abelian groups.

\begin{thm} \label{thm:autoequiv} Let $\Theta$ be a monoidal autoequivalence of the (original, $\Z$-graded) category $\HC$ which fixes all objects. Then $\Theta = \Theta_{\chi}$ for
some character $\chi : (\Lambda \times \Gamma) / I \to \Bbbk^\times$, where $V$ is equipped with some $\Gamma$-valued $W$-invariant grading, and $(\Lambda \times \Gamma) / I$ is the
universal grading group from Theorem \ref{thm:mostgeneral}. \end{thm}

By composing with a color-swap autoequivalence coming from a Dynkin diagram, this theorem also classifies autoequivalences which permute the generating objects $\{B_s\}$.

To give an example of the utility of this theorem, consider the following corollary.

\begin{cor} \label{cor:identitycriterion} Let $\Theta$ be a $\Z$-graded monoidal autoequivalence of $\HC$ which \begin{enumerate}
	\item fixes $B_s$ for each $s \in S$, 
	\item fixes polynomials, and
	\item fixes $\finaldotred$, the generator of $\Hom(B_s,\1)$ for each $s \in S$. \end{enumerate}
Then $\Theta$ is equal to the identity functor. \end{cor}

\begin{proof} By the previous theorem, $\Theta = \Theta_{\chi}$ for some $\chi$.  Fixing $\finaldotred$ implies that $\chi(f_s) = 1$ for all $s \in S$. Fixing polynomials implies that $\chi(\Gamma) = 1$. Therefore $\chi(f_s + g_s) = 1$, so $\chi(g_s) = 1$. Thus $\chi$ sends the entire grading group to $1 \in \Bbbk^{\times}$, and $\Theta_{\chi}$	is the identity functor. \end{proof}

An quick consequence of Corollary \ref{cor:identitycriterion} is the following statement about functors out of the Hecke category.

\begin{cor} \label{cor:isomcriterion}  Let $F$ and $G$ be monoidal $\Z$-graded fully-faithful functors from $\HC$ to a graded additive monoidal category $\CC$, for which \begin{enumerate}
	\item we have chosen an isomorphism $\phi_1 \colon F(\one) \simto G(\one)$ which respects the monoidal identity structure,
	\item we have chosen isomorphisms $\phi_s \colon F(B_s) \simto G(B_s)$ for each $s \in S$, 
	\item $F$ and $G$ agree on polynomials, after intertwining with $\phi_1$, and
	\item $F$ and $G$ agree on $\finaldotred$ for each $s \in S$, after intertwining with $\phi_s$ and $\phi_1$. \end{enumerate}
Then $F \cong G$ via the unique monoidal natural transformation extending $\phi_1$ and $\{\phi_s\}$. \end{cor}

\begin{proof}(Sketch) A fully faithful functor can be inverted on its essential image, so, loosely speaking, we apply Corollary \ref{cor:identitycriterion} to $F^{-1} \circ G$. A more careful proof appears in \S\ref{s:proofs}. \end{proof}

Both corollaries are tools which allow one to prove results about autoequivalences or functors with the minimum of computational effort. For example, using Corollary
\ref{cor:identitycriterion} we can avoid computing how $\Theta$ acts on most of the generators of the Hecke category, including the complicated $2m$-valent vertices.

\begin{rem} Corollary \ref{cor:isomcriterion} will be applied in a followup paper \cite{EHBraidToolkit} to give quick proofs of folklore results about conjugation by braids. Let $K^b(\HC)$ denote the category of bounded chain complexes in $\HC$, with morphisms considered modulo homotopy. For example, it is a folklore theorem for $W =
S_n$ that the (Rouquier complex of the) full twist braid $FT$ is in the Drinfeld center of the Hecke category, i.e. conjugation by the full twist is isomorphic to the identity
functor of $K^b(\HC)$. Using Corollary \ref{cor:isomcriterion}, one can efficiently prove that the two functors $\HC \to K^b(\HC)$ below, \begin{equation} X \mapsto X \qquad \text{
and } \qquad X \mapsto FT \otimes X \otimes FT^{-1}, \end{equation} are isomorphic.  It is less well-known that conjugation by the half twist braid is equivalent to the composition of a Dynkin diagram automorphism, and a degree-based rescaling $\Theta_{\chi}$. Here we use the bigrading of Theorem \ref{thm:bigraded}, and $\chi(a,b) = (-1)^b$. \end{rem}

\begin{rem} The previous remark treats isomorphisms between functors $\HC \to K^b(\HC)$. These functors extend to functors $K^b(\HC) \to K^b(\HC)$, though it does not automatically follow that the natural isomorphism of functors will extend to the homotopy category. We prove this extension result in \cite{EHBraidToolkit} as well. \end{rem}

\section{Proofs} \label{s:proofs}

A presentation of $\HC$ can be found in \cite[\S 10.2.4]{EMTW}. All equation references of the form (8.X) or (10.X) refer to \cite{EMTW}. In this section we prove Theorem
\ref{thm:multigraded}, Theorem \ref{thm:mostgeneral}, Theorem \ref{thm:autoequiv}, and Corollary \ref{cor:isomcriterion} simultaneously. The arguments are parallel, the only difference is the setup. Corollary \ref{cor:isomcriterion} has the most distracting setup, so we postpone its discussion until the end.

Suppose we have a new grading on $\HC$ by an abelian group which refines the original grading. Since $\startdotred$ lives in a one-dimensional graded morphism space, it must be
homogeneous for the new grading. Write $f_s$ for its new degree. Similarly, let $\Theta$ be a monoidal autoequivalence which fixes objects. Therefore $\Theta$ sends $\startdotred$
to some scalar multiple of $\startdotred$. Write $\kappa_s$ for this scalar.

By the same argument, each of the other one-color generators of $\HC$ is homogeneous for the new grading, and is sent by $\Theta$ to scalar multiples of themselves. Write $g_s$ for
the new degree of $\finaldotred$, and $\lambda_s$ for the scalar under $\Theta$. By examining the unit and counit relations (8.5ab) we see that the degree of the split must be
$-g_s$, and the degree of the merge must be $-f_s$, as in \eqref{eq:onecolor}. Similarly, the scalar on the split must be $\lambda_s^{-1}$, and the scalar on the merge must be
$\kappa_s^{-1}$.

In particular, by the barbell relation (10.7c), $\alpha_s$ has degree $f_s + g_s$ and is rescaled by $\kappa_s \lambda_s$. Pick $s \ne t$. In order for the polynomial forcing relation
(10.7d) to be homogeneous or be preserved by $\Theta$, when we set $f = \alpha_t$ (in the notation of (10.7d)), we deduce that $f_s + g_s = f_t + g_t$ whenever $\pa_s(\alpha_t) \ne 0$. Similarly, we deduce that $\kappa_s \lambda_s = \kappa_t \lambda_t$ whenever $\pa_s(\alpha_t) \ne 0$.

The reader should now see that the arguments involving $\Theta$ and the rescaling parameters like $\kappa_s$ are precisely analogous to the arguments involving homogeneity and degrees
like $f_s$. We now omit the discussion of scalars, noting that Theorem \ref{thm:autoequiv} is a direct consequence of this parallel behavior, once the classification of gradings is
shown.

Let us also address general elements of $V$, as in Theorem \ref{thm:mostgeneral}. Since $V$ is homogeneous with respect to the original grading, it must split into homogeneous
components with respect to the new grading. Choose $v \in V$ homogeneous for the new grading. In order for (10.7d) to be homogeneous, we need $\deg(s(v)) = \deg(v)$. Consequently, the
grading on $V$ is $W$-invariant. If $\pa_s(v) = 0$, then $s(v) = v$ and (10.7d) is homogeneous. If $\pa_s(v) \ne 0$ then $s(v) = v - \pa_s(v) \alpha_s$, and the $W$-invariance of the
grading implies that $\deg(v) = \deg(\alpha_s)$. Now one can confirm that (10.7d) is homogeneous.

Continuing, note that the $s$-colored cap has degree $g_s - f_s$, and the cup has degree $f_s - g_s$. It is straightforward to verify that the isotopy relations (10.14abc) are
homogeneous. Now the remaining one-color relations can be stated up to isotopy as in (10.7a-e), and are clearly homogeneous.

Pick $s \ne t$ with $m_{st} < \infty$. By the same argument as above, the $2m_{st}$-valent vertex is homogeneous for the new grading. We set
\begin{equation} \deg \left( \sbotttop \right) = h_{s,t}. \end{equation}
A priori, it is possible that $h_{s,t} \ne h_{t,s}$.

When $m_{st}$ is even, the two-color associativity relation (see (10.7fh)) is homogeneous if and only if $h_{s,t} - g_s = 2 h_{s,t} - g_s$, or equivalently, $h_{s,t} = 0$. When
$m_{st}$ is odd, the two-color associativity relation (see (10.7g)) is homogeneous if and only if $h_{t,s} - g_s = 2h_{t,s} - g_t$, or equivalently, $h_{t,s} = g_t - g_s$. Thus the
degree $h_{s,t}$ agrees with \eqref{eq:2mvalent}. It is now easy to confirm the homogeneity of the isotopy relation (10.14d) and the Zamolodchikov relations (see (10.7mn)).

It remains to discuss the Jones-Wenzl relation (10.7i). We claim that the Jones-Wenzl idempotent (and each term therein) is homogeneous of graded degree zero, which suffices to prove that (10.7i) is homogeneous. The case $m_{st} = 2$ is easy. Any term in the Jones-Wenzl idempotent is the deformation retract of a (two-colored) crossingless matching in some Temperley-Lieb algebra, see \cite[\S 5.3.2]{ECathedral}. Cups and caps deformation retract to pitchforks of various colors. When $m_{st} > 2$, we know that $f_s + g_s = f_t + g_t$, from which we deduce
\begin{subequations}
\begin{equation} \deg(\pitchred) = f_t - g_s = \deg(\pitchblue) = f_s - g_t, \end{equation}
\begin{equation} \deg(\pitchdownred) = g_t - f_s = \deg(\pitchdownblue) = g_s - f_t. \end{equation}
\end{subequations}
Thus the degree of (the deformation retract of) a cap is opposite that of a cup. Any crossingless matching with the same number of boundary points on top and bottom has the same number of cups as caps, so the overall degree of the deformation retract in $\HC$ is zero.

We have now confirmed that the homogeneity of all relations in $\HC$ is equivalent to the grading being defined as in Theorem \ref{thm:mostgeneral}. This concludes the proof, modulo the following remark.

It is conceivable that $\HC$ might admit a grading refining the original $\Z$-grading, but for which the relations of (10.7) and (10.14) were not homogeneous. If so, each
homogeneous component of a non-homogeneous relation would give a new equality. For example, suppose that (10.7h) was not homogeneous. We have already argued that each of the
diagrams appearing must be homogeneous of some degree, so if their degrees did not agree, then the new equalities would state that each diagram is individually zero. This would
clearly impose a nontrivial quotient upon $\HC$, a contradiction. By examination, one quickly determines that each of the relations in the presentation of \cite{EMTW} must be
homogeneous in order for $\HC$ (and not some quotient of $\HC$) to admit a grading.

Finally, let us adapt this proof to the setup of Corollary \ref{cor:isomcriterion}. We remind the reader that, in applications, $F$ and $G$ might be functors like conjugation by an
invertible Rouquier complex $T$. So $F(\one) = T \ot T^{-1}$ is isomorphic to $\one$, but is not the same object. Part of what it means for $F$ to be a monoidal functor is that
$F(\one)$ can be equipped with the structure of a monoidal identity, or equivalently, that it can be made isomorphic to $\one$ in a fashion compatible with the structure maps of the
monoidal identity. Moreover, we assume $\phi_1$ intertwines this structure for $F$ with that of $G$, and leave the details to the reader.

For any object $X$ of $\HC$, let $\phi_X \co F(X) \to G(X)$ be the unique map which monoidally extends the maps $\phi_1$ and $\{\phi_s\}$. For example, if $X = B_s B_t$ then $\phi_X = \phi_s \ot \phi_t$. Any object is a tensor power of $\one$ and various $B_s$, so this map $\phi_X$ makes sense. Let $\Theta_{X,Y} \co \Hom_{\HC}(X,Y) \to \Hom_{\HC}(X,Y)$ be the unique linear map making the following diagram commute.
\begin{equation}
	\begin{tikzcd}[column sep = huge]
	\Hom_{\HC}(X,Y) \arrow[d,"F"] \arrow[r,"\Theta_{X,Y}"] & \Hom_{\HC}(X,Y) \arrow[d,"G"] \\
	\Hom_{\CC}(F(X),F(Y)) \arrow[r,"\phi_Y \circ (-) \circ \phi_X^{-1}"] & \Hom_{\CC}(G(X),G(Y))
	\end{tikzcd}
\end{equation}

By assumption, $\Theta_{\one, \one}$ is the identity map on polynomials, and $\Theta_{B_s,\one}$ fixes the dot. Note that $\Theta_{\one,B_s}$ is a $\Bbbk$-linear map between free rank 1 $\Bbbk$-modules, so it must rescale the other dot by some factor $\kappa_s \in \Bbbk^{\times}$. One can continue the argument as above. Alternatively, one can observe that $\Theta_{X,Y}$ is the action on morphism spaces of some autoequivalence $\Theta$ which fixes objects. Then one can apply Corollary \ref{cor:identitycriterion} to finish the proof.

\bibliographystyle{plain}
\bibliography{mastercopy}

\end{document}